\documentclass[12pt]{article}
\usepackage[latin1]{inputenc}
\usepackage{amsmath}
\usepackage{amsfonts}
\usepackage{amssymb}
\usepackage{booktabs}
\usepackage{multirow}
\usepackage[american]{babel}
\usepackage{graphicx}
\usepackage{amsthm}
\usepackage{color}
\usepackage{xcolor}
\usepackage[normalem]{ulem}
\usepackage[text={5in,7.125in},centering]{geometry}
\usepackage{hyperref}
\usepackage{caption}
\usepackage{float}

\begin{document}


\title{A characterization of graphs with small palette index}
\date{}

\newtheorem{theorem}{Theorem}
\newtheorem{proposition}[theorem]{Proposition}
\newtheorem{lemma}[theorem]{Lemma}
\newtheorem{definition}[theorem]{Definition}
\newtheorem{example}[theorem]{Example}
\newtheorem{corollary}[theorem]{Corollary}
\newtheorem{conjecture}[theorem]{Conjecture}
\newtheorem{remark}{Remark}
\newtheorem{problem}{Problem}
\newtheorem{claim}{Claim}

\author{Domenico Labbate\footnote{Dipartimento di Matematica, Informatica ed Economia, Universit\`{a} degli Studi della Basilicata, Italy.}, Davide Mattiolo\footnote{Department of Computer Science, KU Leuven Kulak, Belgium.}, Giuseppe Mazzuoccolo \footnote{Dipartimento di Informatica, Universit\`{a} degli Studi di Verona, Italy.}, \\Federico Romaniello$^{*}$, Gloria Tabarelli \footnote{Dipartimento di Matematica, Universit\`{a} di Trento,
Italy.}}
\maketitle

\begin{abstract}
Given an edge-coloring of a graph $G$, we associate to every vertex $v$ of $G$ the set of colors appearing on the edges incident with $v$. The palette index of $G$ is defined as the minimum number of such distinct sets, taken over all possible edge-colorings of $G$. A graph with a small palette index admits an edge-coloring which can be locally considered to be almost symmetric, since few different sets of colors appear around its vertices.
Graphs with palette index $1$ are $r$-regular graphs admitting an $r$-edge-coloring, while regular graphs with palette index $2$ do not exist. Here, we characterize all graphs with palette index either $2$ or $3$ in terms of the existence of suitable decompositions in regular subgraphs. As a corollary, we obtain a complete characterization of regular graphs with palette index ~$3$.
~
Math. Subj. Class.: \textit{05C15}
\end{abstract}

\section{Introduction}\label{sec:intro}
A classical and well-known problem in graph theory is to find a coloring of the edges of a graph, using as few colors as possible, such that no two incident edges receive the same color.
 In this paper, we consider a related problem, first introduced in \cite{meszka}, where the problem of minimizing the number of different sets of colors appearing on the edges around each vertex is proposed. More precisely, a \emph{$k$-edge-coloring} of a graph $G$ is a map $c\colon E(G)\to \{1,2,...,k\}$ such that for any pair of incident edges $e_1$ and $e_2$, $c(e_1)\neq c(e_2)$. If a graph $G$ has maximum degree $\Delta$ and it admits a $\Delta$-edge-coloring, then it is said to be a \emph{Class 1} graph, otherwise it is called a \emph{Class 2} graph. A Class 1 regular graph admits an edge-coloring with the same set of colors around each vertex and hence, to some extent, it locally admits a very symmetric edge-coloring. We denote by $\partial_G (v)$ the set of edges incident with a vertex $v \in V(G)$ and we define the \emph{palette} of $v$ with respect to an edge-coloring $c$ as $P_c(v)=\{c(e) \mid e \in \partial_G (v)\}$. Note that if the degree of a vertex $v \in V(G)$ is $d_G(v)=0$, i.e.\ if $\partial_G (v)= \emptyset$, then $P_c(v)= \emptyset$. The \emph{palette index}, denoted by $\check{s}(G)$ {{(see \cite{meszka})}}, of a graph $G$ is the minimum number of distinct palettes, taken over all edge-colorings, occurring among the vertices of the graph. Graphs with small palette index are in some sense close to be Class 1, since they admit an edge-coloring with few different sets of colors around each vertex, that is an edge-coloring which can be locally considered almost symmetric. 

The palette index of a graph was first introduced in \cite{meszka} for simple graphs, but the generalization to graphs with multiple edges is straightforward (see \cite{AveBonMaz}). 
It is an easy consequence of the definition that a graph $G$ has $\check{s}(G)=1$ if and only if it is a Class 1 regular graph. 
In paper \cite{meszka}, the authors determine the palette index of complete graphs, and they furnish the following further results in the regular case:

\begin{lemma}[\cite{meszka}] \label{lemma:not2}
Let $G$ be a regular graph. Then, $\check{s}(G) \neq 2$.
\end{lemma}

\begin{theorem}[\cite{meszka}]\label{thm:cubic}
Let $G$ be a connected cubic graph. Then, 
\begin{itemize}
		\item $G$ is Class $1$ if and only if $\check{s}(G)=1$;
		\item $G$ is Class $2$ with a perfect matching if and only if $\check s(G)=3$;
		\item $G$ is Class $2$ without a perfect matching if and only if $\check s(G)=4$.
\end{itemize}
\end{theorem}

Further results on the palette index of a graph appeared since then (see, e.g., \cite{AveBonMaz, BonMaz, CasPet, Gha, spanning_star, HorHud, Smb}). Many of them consider the computation of the palette index for some specific classes of graphs, such as complete bipartite graphs, $4$-regular graphs and some others.
 
On the other hand, very few results on the palette index of general graphs appeared in the literature until now. A recent one is proved by some of the authors in \cite{palette_large}:

\begin{theorem}[\cite{palette_large}]\label{thm:lowerbound}
Let $G$ be a graph and let $\Delta$ and $\delta$ denote its maximum and minimum degree, respectively. Suppose that $\Delta \geqslant 2$ and $G$ has no spanning even subgraph without isolated vertices. Then, $\check s(G) > \delta $.
\end{theorem}
 
Note that both Theorems \ref{thm:cubic} and \ref{thm:lowerbound} give conditions for a graph $G$ to have a restriction on the palette index value in terms of some structural properties of $G$. Analogous results are proved for the $4$-regular case in \cite{BonMaz}.

In this paper, in the very same spirit, we consider the problem of relating the palette index of a graph $G$ to some structural properties of $G$. 
In Section \ref{definitions}, we introduce a description of the set of palettes induced by an edge-coloring of a graph in terms of an associated hypergraph $H$. This description turns out to be very practical in proving our main results in Section \ref{main}, where we present a complete characterization of graphs with palette index at most $3$ in terms of the existence of some graph decompositions into Class 1 regular subgraphs.

\section{Notation and preliminary results}\label{definitions}

In the course of the paper, we will make use both of graphs and hypergraphs, where a graph is nothing but a hypergraph with all edges of cardinality two. Every time we refer to a hypergraph, we mean that it may admit both parallel hyperedges (i.e., hyperedges on the same subset of vertices) and loops (i.e., hyperedges of cardinality one), whereas we consider only graphs without loops. 
For basic notation not defined here, we refer to \cite{BondyMurty}.

Recall that a \emph{decomposition} of a graph $G$ is a family $\{H_i\}_{i \in I}$ of subgraphs of $G$ such that $E(H_i)\neq \emptyset$ for every $i \in I$, $\bigcup_{i \in I}E(H_i)=E(G)$ and $E(H_i)\cap E(H_j)=\emptyset$ for every $i \neq j \in I$.
 
We denote by $\mathcal{C}=\{1,2,...,k\}$ the set of colors and by $\mathcal{P}_c=\{P_1,P_2,...,P_t\}$ the set of distinct palettes that $c$ induces among the vertices of $G$. Observe that the empty palette belongs to the set $\mathcal{P}_c$ if and only if $G$ has some isolated vertices. For each color $i \in \mathcal{C}$ we define $E_i=\{e \in E(G) \mid c(e)=i\}$, and for each palette $P_j \in \mathcal{P}_c$ we define $V_j=\{v \in V(G) \mid P_c(v)=P_j\}$. Finally, if $\emptyset \neq X \subseteq \mathcal{C}$, we denote by $G[X]$ the subgraph of $G$ induced by all the edges $e \in E(G)$ such that $c(e) \in X$.

\begin{remark}\label{induced_regular_class_1} 
Let $\emptyset \neq X \subseteq \mathcal{C}$. The following statements are equivalent:
\begin{itemize}
\item $G[X]$ is an $|X|$-regular Class 1 subgraph of $G$
\item for all $ v \in V(G[X])$, $X \subseteq P_c(v)$.
\end{itemize}
\end{remark}


The following definition will be largely used in what follows.

\begin{definition}
A $k$-edge-coloring $c$ of $G$ is \emph{$\check{s}$-minimal} if its associated set of palettes $\mathcal{P}_c$ has cardinalty $\check{s}$ and there is no $k'$-edge-coloring $c'$ of $G$ with an associated set of palettes such that $|\mathcal{P}_{c'}|=| \mathcal{P}_c|$ and $k'<k$.
\end{definition}

In other words, an edge-coloring of a graph $G$ is $\check{s}$-minimal if it has the minimum number of colors among all edge-colorings minimizing the number of palettes. Such a parameter was first considered in \cite{AveBonMaz}.
Moreover, it is remarked in \cite{meszka} that an $\check{s}$-minimal edge-coloring could need a number of colors larger than the chromatic index of \mbox{the graph.}

It will be practical in what follows to associate a hypergraph to an edge-coloring of a graph $G$. Hence, let $c$ be a $k$-edge-coloring of a graph $G$. We define $H_G^c$ as the hypergraph with $\mathcal{P}_c$ as set of vertices and $k$ hyperedges $h_1,h_2,...,h_k$, where $h_i=\{P \in \mathcal{P}_c \mid i \in P\}$ (see the top of Figure \ref{fig:ipergrafo_associato} for an example). Moreover, when (almost) all hyperedges of $H_G^c$ have size 1 or 2, we will depict them as loops and edges of a multigraph (see the bottom of Figure \ref{fig:ipergrafo_associato}). 

\begin{figure}[H]
\includegraphics[scale=0.5]{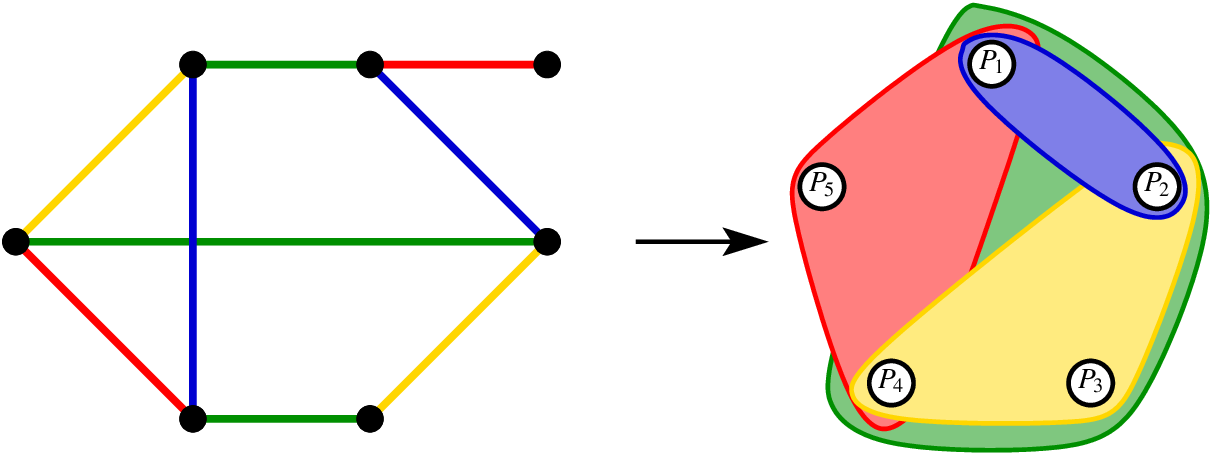}
\includegraphics[scale=0.5]{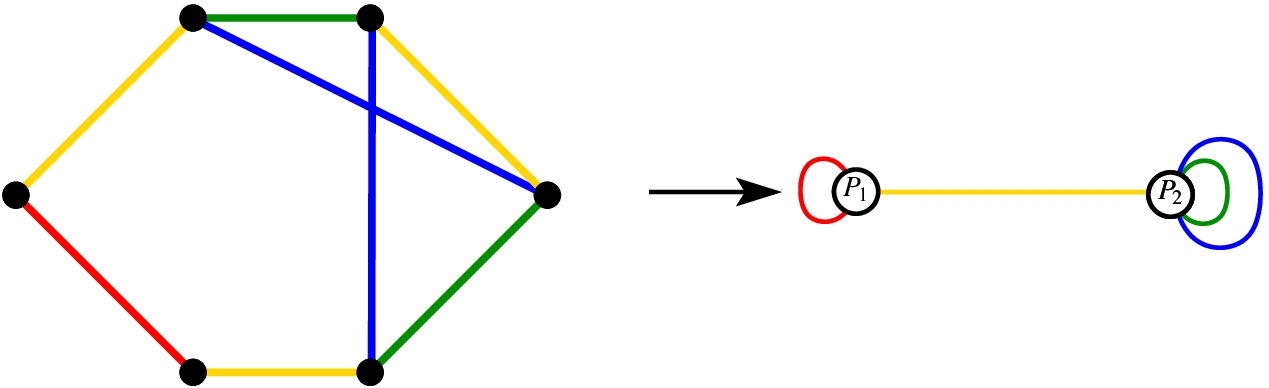}
\caption{Edge-colorings and their associated hypergraphs.}\label{fig:ipergrafo_associato}
\end{figure}

Observe that if $c$ is an $\check{s}$-minimal edge-coloring of $G$, and $h_\alpha,h_\beta \in E(H_G^c)$, then there must exist a vertex $v \in V(H_G^c)$ such that both $h_\alpha$ and $h_\beta$ are incident with $v$. For otherwise, if $h_\alpha$ and $h_\beta$ are independent hyperedges in $H_G^c$, there does not exist a palette $P \in \mathcal{P}_c$ with $\{\alpha, \beta\} \subseteq P$. Hence, the subgraph of $G$ induced by  $\{e \in E(G) \mid c(e) \in \{\alpha, \beta\}\}$ is $1$-regular. Therefore, it is possible to recolor the edges that receive color $\beta$ in $c$ with $\alpha$, thus constructing a new edge-coloring of $G$ with the minimum number of palettes $\check{s}$ and less than $k$ colors, contradicting the $\check{s}$-minimality of $c$. Hence we have the following:

\begin{remark} \label{min_hyper_property}
If $c$ is an $\check{s}$-minimal edge-coloring of $G$, then any two hyperedges $h_i, h_j \in E(H_G^c)$ are incident with some common vertex.
\end{remark}

{{Although it is not needed for proving our main results,}} we conclude this section by showing an interesting relation between the hypergraph $H_G^c$ and the concept of $H$-coloring of a graph $G$. 
Recall that, if $G$ and $H$ are graphs, an $H$-coloring of $G$ is an edge-coloring $f \colon E(G) \to E(H)$ such that for every vertex $u \in V(G)$ there exists a vertex $v \in V(H)$ with $f(\partial_G(u))=\partial_H(v)$. 
The concept of $H$-coloring is mainly known in relation to the Petersen coloring Conjecture by Jaeger (see \cite{FJaeger-nzf_problems}), which states that every bridgeless cubic graph admits a $P$-coloring, where $P$ is the Petersen graph. The conjecture is largely open, and it is known that it implies many other well-known conjectures such as the Cycle Double Cover Conjecture and the Berge-Fulkerson Conjecture (see \cite{MazMkr} for more details on such conjectures and implications).

There is no obstruction to define $\partial_H(v)$ as the set of all hyperedges incident with the vertex $v$ of a hypergraph $H$. Then, we can also assume $H$ to be a hypergraph in the definition of $H$-coloring. Although we are not aware of any paper where this more general formulation is considered, it naturally works in the same way.

In our context, it is straightforward that the map $f\colon E(G) \to E(H_G^c)$, such that $f(e)=h_{c(e)}$, is an $H_G^c$-coloring of $G$. Then, an alternative description of the palette index of a graph directly follows.

\begin{proposition}
The palette index $\check{s}(G)$ of a graph $G$ is equal to the order of a smallest hypergraph $H$ such that $G$ admits an $H$-coloring.
\end{proposition}

\section{Graphs with palette index at most 3} \label{main}

In this section, we provide a complete characterization of graphs with palette index at most $3$. First of all, in Theorem \ref{index_two}, we characterize graphs with palette index equal to $2$ and then, in Theorem \ref{thm:atmost3}, we characterize graphs with palette index at most $3$. Both characterizations are given with respect to the existence of a suitable decomposition of the graph $G$ in Class 1 regular subgraphs.
Note that the number of such subgraphs depends on the graph $G$ itself and, in particular, it is not uniquely determined by its palette index. In what follows we consider families of graphs with given palette index. Therefore, we need to consider decompositions in ``at most'' a certain number of subgraphs $H_i$'s, where by at most we mean that not all $H_i$'s always appear in the decomposition of a certain $G$ in the considered family.
In order to improve the readability of the following proofs, we always describe each subgraph $H_i$, but we omit to specify every time in what instances such an $H_i$ is actually present in the decomposition.

\begin{theorem}\label{index_two}
Let $G$ be a graph, and let $\Delta$ and $\delta$ denote its maximum and minimum degree, respectively. 
Then, $\check{s}(G)=2$ if and only if $\Delta > \delta$  and  $G$ can be decomposed into at most two Class 1 regular subgraphs $H_0$ and $H_1$ such that \begin{itemize}
\item[-] $H_0$ is spanning and $\delta $-regular;
\item[-] $H_1$ is $(\Delta-\delta )$-regular.
\end{itemize}
\end{theorem}
\begin{proof}
Let $G$ be a graph with palette index $2$. We show that $G$ has the required decomposition.

By Lemma \ref{lemma:not2} $G$ is not regular, therefore $\Delta>\delta$. Let $c\colon E(G)\to \{1,2,...,k\}$ be an $\check{s}$-minimal edge-coloring of $G$, with $\mathcal{P}_c=\{P_1,P_2\}$. Its associated hypergraph $H_G^c$ is of order two and may admit only two types of hyperedges: hyperedges both incident with $P_1$ and $P_2$, and loops incident with a unique palette. Since $c$ is $\check{s}$-minimal, by Remark  \ref{min_hyper_property}, all loops must be incident with the same palette, say $P_1$. 
If there is no hyperedge in $H_G^c$ with cardinality two, then $P_2$ is the empty palette, meaning that $G$ has some isolated vertices, that is $\delta=0$. Since $\mathcal{C} \neq \emptyset$, we define $H_1=G[\mathcal{C}]$. Since $V(G[\mathcal{C}])=V_1$ and $\mathcal{C}=P_1$, by Remark \ref{induced_regular_class_1}, $H_1$ is a $\Delta$-regular Class 1 subgraph of $G$. Clearly $E(H_1)= E(G)$ so that $\{H_1\}$ is a decomposition of $G$ as required in the statement.
If $H_G^c$ has hyperedges $h_1,\dots,h_t$,  $t>0$, of cardinality two with vertices $P_1$ and $P_2$, without loss of generality we denote by $h_{t+1},\ldots,h_k$ the remaining loops on $P_1$. Observe that $k>t$, for otherwise, we would have that $P_1=P_2$, so that $\check{s}(G)=1$ and $|V(H_G^c)|=1$ since $c$ is $\check{s}$-minimal.
The hypergraph $H_G^c$ is depicted in Figure \ref{hyper_index_2}.
\begin{figure}[H]
\includegraphics[width=7cm]{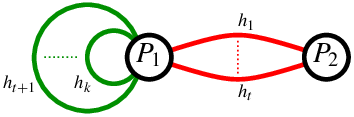}
\caption{The hypergraph $H_G^c$ associated to a $\check{s}$-minimal edge-coloring of a graph with $\check{s}(G)$ = 2.}
\label{hyper_index_2}
\end{figure}
Observe that $P_2=\{1,2,...,t\} \subset P_1=\{1,2,...,k\}$. Hence $V(G[P_2])=V_1 \cup V_2 =V(G)$ and for all $v \in V(G)$, $P_2 \subset P_c(v)$. Set $H_0=G[P_2]$. By Remark \ref{induced_regular_class_1}, the subgraph $H_0$ is a spanning $t$-regular Class 1 subgraph of $G$. Moreover, $t=\delta$.
Let $X=\mathcal{C} \setminus P_2=\{t+1,...,k\}$ and set $H_1=G[X]$. It follows $V(H_1)=V_1$, and, for every $v \in V(G)\cap V_1$, $X \subseteq P_1$. By Remark \ref{induced_regular_class_1}, $H_1$ is a $(k-t)$-regular Class 1 subgraph of $G$. Moreover, $k=\Delta$ and then $k-t=\Delta-\delta$ holds. Since $\mathcal{C}= P_2 \cup X $, $E(H_0)\cup E(H_1)=E(G)$ follows. Moreover, $P_2 \cap X=\emptyset$, so that $\{H_0,H_1\}$ is a decomposition of $G$ as required in the statement.  

Conversely, assume to have a decomposition of $G$ into at most two subgraphs $H_0$ and $H_1$ as in the statement. Since $G$ is not regular, then $\check{s}(G)\geq 2$. It remains to prove that $\check{s}(G)\leq 2$. To this aim, observe that, if $\delta>0$, $H_0$ is a $\delta$-regular Class 1 subgraph and then it admits a $\delta$-edge-coloring $c_0$ with colors $\{1,2,...,\delta\}$. For the same reason, $H_1$ is $(\Delta-\delta)$-regular and admits a $(\Delta-\delta)$-edge-coloring $c_1$ with colors $\{\delta+1,...,\Delta\}$. Define the edge-coloring $c\colon E(G)\to \{1,2,...,\Delta(G)\}$ as $$c(e)=\begin{cases}
c_0(e) \text{ if } e \in E(H_0),\\
c_1(e) \text{ if } e \in E(H_1).
\end{cases}$$

For every vertex $v \in V(G)$, $\partial_G(v)= \partial_{H_0}(v) \cup \partial_{H_1}(v)$. Hence, either $P_c(v)=\{1,2,...,\delta\}$ or $P_c(v)=\{1,2,....,\delta,\delta +1,...,\Delta\}$, i.e.,\ $\check{s}(G)\leq 2$.
\end{proof}

The next result gives a complete characterization of graphs with palette index at most $3$.  

\begin{theorem}\label{thm:atmost3}
Let $G$ be a graph. Then $\check{s}(G)\leq 3$ if and only if there exists a decomposition of $G$ in at most four Class 1 regular subgraphs $H_0, H_1, H_2, H_3$ with the following properties:
\begin{itemize}
\item $H_0$ is spanning
\item there exists a partition of $V(G)$ in at most three subsets $A_1,A_2, A_3$ such that $V(H_1)=A_2 \cup A_3$, $V(H_2)=A_1 \cup A_3$ and either $V(H_3)=A_3$ or $V(H_3)=A_1 \cup A_2$.
\end{itemize}

\end{theorem}
\begin{proof}
Assume that $\check{s}(G) \leq 3$ and let $c$ be an $\check{s}$-minimal coloring of $G$.

If $\check{s}=1$, $G$ is Class 1 and regular. If $E(G)= \emptyset$, each vertex is associated with the empty palette. Let $E(G)\neq \emptyset$, then we set $H_0=G$, $A_1=V(G)$. Hence $\{H_0\}$ is the required decomposition of $G$.

If $\check{s}=2$, then $G$ is non regular with palettes $P_1$ and $P_2$. By Theorem \ref{index_two}, $G$ admits a decomposition into at most two subgraphs $H_0$ and $H_1$. Let $A_1=A_3=\emptyset$ and $A_2=V_1$. Then, $\{H_0,H_1\}$ is the required decomposition. 


If $\check{s}=3$, let $\mathcal{P}=\{P_1,P_2,P_3\}$. Up to the existence of some colors which belong to all the three palettes, {{by Remark \ref{min_hyper_property}}} the only possibilities for the associated hypergraph $H_G^c$ are the four depicted in Figure \ref{fig:hyp3}.

\begin{figure}[H]
\includegraphics[width=7cm]{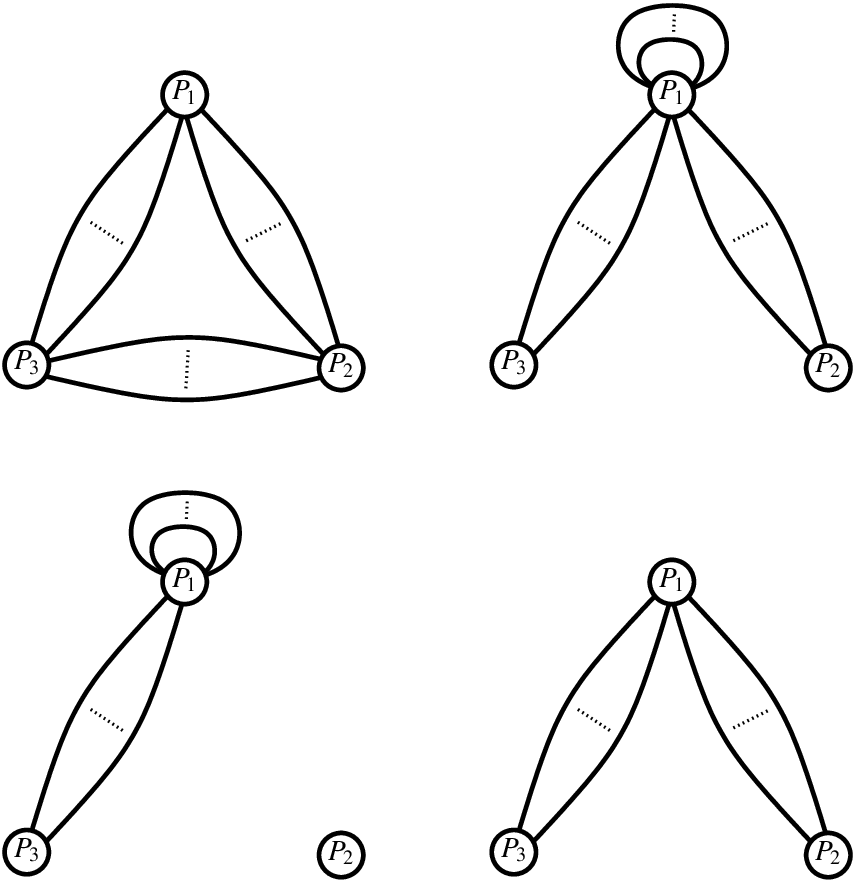}
\caption{The four possibilities for the associated hypergraph $H_G^c$ when $\check{s}=3$ (up to relabelling palettes).}
\label{fig:hyp3}
\end{figure}

We set $A_i=V_i$, for all $i=1,2,3$.
Since $c$ is $\check{s}$-minimal, for every $i\neq j\neq k \in \left\{1,2,3\right\}$ the set of colors lying in $(P_i \cap P_j) \setminus P_k$ and the set of colors lying in $P_k \setminus (P_i \cup P_j)$ cannot be both non-empty.
In addition to this, among the three sets $P_k \setminus (P_i \cup P_j),~i\neq j\neq k$, at most one can be non-empty.
If $P_1\cap P_2 \cap P_3 \neq \emptyset$, set $H_0=G[P_1\cap P_2 \cap P_3]$. Since $P_1\cap P_2 \cap P_3 \subseteq P_c(v)$ for all $v \in V(G)$, $H_0$ is regular Class 1 and spanning, by Remark \ref{induced_regular_class_1}.

We must distinguish two cases:
\begin{itemize}
	\item $P_k \setminus (P_i \cup P_j) = \emptyset$, for all $i\neq j \neq k \in \{1,2,3\}$
\end{itemize}
In this case, if $(P_i \cap P_j) \setminus P_k \neq \emptyset$, set $H_k=G[(P_i \cap P_j) \setminus P_k]$, for all $i \neq j \neq k \in \{1,2,3\}$. Please note that the non-empty sets $(P_i \cap P_j) \setminus P_k$, together with $P_1\cap P_2\cap P_3$, form a partition of $\mathcal{C}$, so that $\{H_0,H_1, H_2, H_3\}$ is a decomposition of $G$. Observe also that $V(H_k)= A_i \cup A_j$ and for all $v \in A_i \cup A_j$, either $P_c(v)=P_i$ or $P_c(v)=P_j$. This implies that $(P_i \cap P_j) \setminus P_k \subseteq P_c(v)$, so that each $H_k$ is Class 1 and regular by Remark \ref{induced_regular_class_1}. Hence $\{H_0,H_1, H_2, H_3\}$ is the required decomposition of $G$. 

\begin{itemize}
	\item $P_k \setminus (P_i \cup P_j) \neq \emptyset$ for some $i\neq j \neq k$
\end{itemize}
Without loss of generality, since at most one among the three sets $P_k \setminus (P_i \cup P_j)$ can be non-empty, suppose that $P_3 \setminus (P_1 \cap P_2) \neq \emptyset$. Then we set:
$H_1=G[(P_2 \cap P_3) \setminus P_1]$ if $(P_2 \cap P_3) \setminus P_1 \neq \emptyset$, $H_2=G[(P_1 \cap P_3) \setminus P_2]$ if $(P_1 \cap P_3) \setminus P_2 \neq \emptyset$ and $H_3=G[P_3\setminus (P_1 \cup P_2)]$. 

Since $\{(P_3 \setminus (P_1 \cup P_2), (P_1\cap P_3) \setminus P_2, (P_2\cap P_3) \setminus P_1, P_1\cap P_2 \cap P_3)\}$ is a partition of $\mathcal{C}$, we conclude that $\{H_0,H_1,H_2,H_3\}$ is a decomposition of $G$. Observe also that $V(H_1)=A_2 \cup A_3$, $V(H_2)=A_1 \cup A_3$ and $V(H_3)=A_3$, and that, by Remark \ref{induced_regular_class_1}, every $H_i$, $i=1,2,3$ is regular and Class 1. Hence $\{H_0,H_1,H_2,H_3\}$ is the required decomposition of $G$.

Vice versa, let $\{H_0,H_1,H_2,H_3\}$ a decomposition of $G$ as in the statement. Let $c_i$ be a minimal edge-coloring of $H_i$, for each $i=0,1,2,3$. We define an edge-coloring $c$ of $G$ as follows. Since for every $e \in E(G)$ there is a unique $i \in \{0,1,2,3\}$ such that $e \in E(H_i)$, let $c(e)=c_i(e)$. It suffices to show that $c$ defines at most three distinct palettes on the graph $G$. By hypothesis $A_1,A_2,A_3$ is a partition of $V(G)$, then for all $ u,v \in A_i$, we have that $u$ and $v$ belong to exactly the same set of subgraphs $H_i$, implying that $P_c(u)=P_c(v)$. Thus, $c$ induces at most three palettes.
\end{proof}

Recall that a graph has palette index $1$ if and only if it is regular and Class 1. Then Theorem \ref{thm:atmost3} in combination with Theorem \ref{index_two} implicitly permits to describe graphs with palette index exactly equal to three.
In particular, a compact description is possible in the regular case.

\begin{corollary}\label{cor:regular_3}
Let $G$ be a $k$-regular graph. Then $\check{s}(G)=3$ if and only if  
$G$ is Class 2 and it can be decomposed in three Class 1 $\frac{k-r}{2}$-regular subgraphs, for $0 \leq r<k$, and, if $r>0$, a Class 1 $r$-regular spanning subgraph.
\end{corollary}
\begin{proof}
Since $G$ is regular, $\check{s}(G)= 1$ if and only if $G$ is a Class 1 graph. Moreover, $\check{s}(G)\neq 2$ by Lemma \ref{lemma:not2}. 
Assume $\check{s}(G)=3$. Then, $G$ is a Class 2 graph. By Theorem \ref{thm:atmost3}, $G$ can be decomposed in at most four Class 1 regular subgraphs $H_0,H_1,H_2$ and $H_3$, where $H_0$, if $r>0$, is a spanning Class 1 $r$-regular subgraph of $G$. Let $A_1,A_2$ and $A_3$ be a partition of $V(G)$ as in the statement of Theorem \ref{thm:atmost3}. The case $V(H_3)=A_3$ cannot occur, otherwise every vertex in $A_3$ belongs to all four subgraphs $H_i$'s, while all vertices in $A_1$ and $A_2$ do not. A contradiction by regularity of $G$ and $H_i$'s.
Then, $V(H_3)=A_1\cup A_2$ holds. Denote by $d_i$ the common degree of every vertex in $H_i$, for $i=1,2,3$.
It follows that for every vertex $v \in A_i$,  
$$k=d_G(v)=r+\sum_{j\neq i} d_j.$$
Hence the following relation holds for each $i=1,2,3$,
$$k-r=\sum_{j\neq i} d_j.$$
It follows that $H_1$,$H_2$ and $H_3$ are all $\frac{k-r}{2}$-regular graphs.
Conversely, the graphs in the statement give a decomposition which satisfies the conditions in the statement of Theorem \ref{thm:atmost3}. Then, $\check{s}(G) \leq 3$. Since $G$ is regular and Class 2, the relation $\check{s}(G)<3$ does not hold and the assertion follows.
\end{proof}

Please note that the spanning subgraph $H_0$ could not appear in the decomposition described in Corollary \ref{cor:regular_3} (or if you prefer it can be considered to be an empty subgraph).  This happens, for instance, in Figure \ref{fig:K7} where we present a decomposition of the complete graph $K_7$ in two copies of $K_4$ and a copy of $K_{3,3}$.

One could be tempted to think that $H_0$ can be always chosen as the largest Class 1 regular subgraph of $G$ in order to obtain a decomposition such as the one described in Corollary \ref{cor:regular_3}.
We conclude the paper by showing that is not the case.
\begin{figure}[H]
\includegraphics[width=5cm]{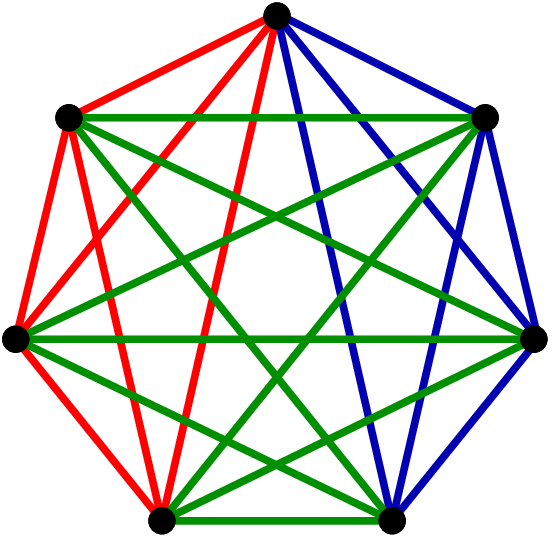}
\caption{A decomposition of $K_7$ in three Class 1 cubic graphs.}\label{fig:K7}
\end{figure} 
In other words, the choice of $H_0$ as the largest Class 1 spanning regular subgraph of $G$ could leave a subgraph of $G$ which does not admit a decomposition in three Class 1 regular subgraphs, while the graph $G$ itself could admit such a decomposition.

We consider the $4$-regular graph depicted in Figure \ref{fig:4reg}. As already shown in \cite{BonMaz}, such graph has a decomposition in three Class 1 $2$-regular subgraphs and then it has palette index equal to $3$. On the other hand, even if it admits a perfect matching, it does not admit two disjoint perfect matchings. Then, the complement of any perfect matching is a cubic graph without a perfect matching, that is by Theorem \ref{thm:cubic} a cubic graph with palette index equal to $4$. It follows that such a cubic graph cannot be decomposed in three Class 1 regular subgraphs as required.

\begin{figure}[H]
\includegraphics[width=5cm]{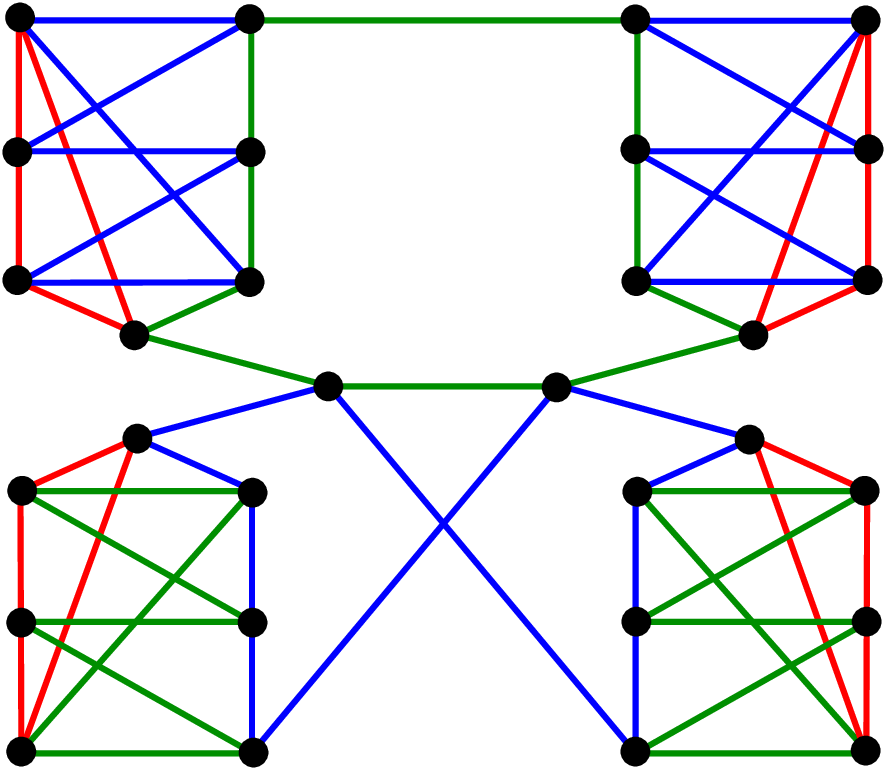}
\caption{A decomposition of a $4$-regular graph in three Class 1 $2$-regular graphs.}\label{fig:4reg}
\end{figure}

\section*{Acknowledgement}
The research that led to the present paper was partially funded by the groups GNSAGA of INdAM and Fondazione Cariverona, program ``\emph{Ricerca Scientifica di Eccellenza 2018}'', project ``Reducing complexity in algebra, logic, combinatorics-REDCOM. Davide Mattiolo is funded by a Postdoctoral Fellowship of the Research Foundation Flanders (FWO), project number 1268323N.

\end{document}